\pdfoutput=1

\documentclass[10pt]{article}
\usepackage[utf8]{inputenc}
\usepackage[T1]{fontenc}
\usepackage{amsmath}
\usepackage{amsfonts}
\usepackage{amssymb}
\usepackage[version=4]{mhchem}
\usepackage{stmaryrd}
\usepackage{bbold}

\usepackage{xspace}
\usepackage{url}
\usepackage{mathtools}
\usepackage{amssymb}
\usepackage{amsthm}
\usepackage{empheq}
\usepackage{latexsym}
\usepackage[shortlabels]{enumitem}
\usepackage{eurosym}
\usepackage{dsfont}
\usepackage{appendix}
\usepackage{color} 
\usepackage[unicode]{hyperref}
\usepackage{frcursive}
\usepackage[utf8]{inputenc}
\usepackage[T1]{fontenc}
\usepackage{geometry}
\usepackage{multirow}
\usepackage{todonotes}
\usepackage{lmodern}
\usepackage{anyfontsize}
\usepackage{stmaryrd}
\usepackage{natbib}
\usepackage{cleveref}
\usepackage[english]{babel}
\usepackage[english=british]{csquotes}
\usepackage{mathtools}
\usepackage{color}
\usepackage{comment}
\usepackage{mdframed}

\usepackage{cases}
\usepackage{bm}
\usepackage[mathscr]{euscript}

\mdfsetup{middlelinecolor=blue,middlelinewidth=2pt,linewidth=0pt,backgroundcolor=blue!15,,roundcorner=10pt}

\allowdisplaybreaks

\setcitestyle{numbers,open={[},close={]}}

\definecolor{red}{rgb}{0.7,0.15,0.15}
\definecolor{green}{rgb}{0,0.5,0}
\definecolor{blue}{rgb}{0,0,0.7}
\hypersetup{colorlinks, linkcolor={red}, citecolor={green}, urlcolor={blue}}
			
\makeatletter \@addtoreset{equation}{section}

\newtheorem{theorem}{Theorem}[section]

\newtheorem{lemma}[theorem]{Lemma}

\newtheorem{definition}[theorem]{Definition}

\newcommand{\vertiii}[1]{{\left\vert\kern-0.25ex\left\vert\kern-0.25ex\left\vert #1 \right\vert\kern-0.25ex\right\vert\kern-0.25ex\right\vert}}

\title{2 Chapter. Random Algebraic construction in extremal graph theory. }

\author{}
\date{}

\begin{document}

\title{Some sharp lower bounds for the bipartite Tur\'{a}n number of theta graphs
}

\author{Stefanos {Theodorakopoulos}\footnote{Department of Mathematics, NTUA, Zografou Campus, 15780 Athens, Greece, steftheodorakopoulos@mail.ntua.gr}}

\date{\today}

\maketitle

\begin{abstract}
We extend Conlon's random algebraic construction
to show that for any odd number $k \geq 3$ exists a natural number $c_k$ (the same as Conlon's) such that $\operatorname{ex}(n^a,n,\theta_{k,c_k}) = \Omega_{k,a}((n^{1 + a})^{\frac{k + 1}{2k}})$, with $a \in [\frac{k - 1}{k + 1}, 1)$. Where given a graph $H$, we denote by $\operatorname{ex}(n,m,H)$ the maximum number of edges an $H-$free bipartite graph can have when the cardinalities of its parts are $n$ and $m$. Also, we denote with $\theta_{k,l}$ the graph where two vertices are connected through $l$ disjoint paths of length $k$.

\vspace{5mm}
\noindent{\bf Key words:} unbalanced; bipartite graph; Tur\'{a}n number; sharp bound; theta graph \vspace{5mm}
\end{abstract}

\section{Introduction}
In exrtemal graph theory, a classical problem is to determine $\operatorname{ex}(n ; H)$ for a given graph $H$, where $\operatorname{ex}(n ; H)$ is defined as the biggest number of edges a graph $G$ with $|V(G)|=n$ not containing a subgraph isomorphic to $H$ can possible have. In practice, the search for lower bounds of the extremal function $\operatorname{ex}(n ; H)$ that are as big as possible has proved to be surprisingly difficult. It is important to note that when someone tries to create an $H$-free graph $G$, the size of $G$ plays a major role to how delicate the said construction needs to be, because, the smaller the size of $G$ the easier it is for $H$ to appear, as the edges have less "space" to move. At great importance is the case where someone has a bipartite graph $G=(A, B)$ with $|A|=n,|B|=m, n \geq m$ and tries to find out how many edges $G$ can have without containing a graph $H$, then we talk about the asymmetric bipartite Tur\'{a}n number $\operatorname{ex}(m, n, H)$ of $H$. A well studied occasion is when $H=C_{2 k}$, the cycle of length $2 k$ for some natural number $k$. The $C_{2 k}$ graphs belong to the more general family of $\theta_{k, \ell}$ graphs, note that $C_{2 k}=\theta_{k, 2}$. Conlon in \cite{conlon2019graphs} by building upon a paper of Bukh \cite{article}, shows that for every natural number $k \geq 2$ there exists a natural number $\ell := \ell(k)$ such that, for every $n$, there is a balanced bipartite graph with $n$ vertices and $\Omega_{k}\big(n^{1+\frac{1}{k}}\big)$ edges with at most $\ell$ paths of length $k$ between any two vertices. A result of Faudree and Simonovits \cite{faudree1983class} implies that the bound on the number of edges is tight up to the implied constant. We extend on this method of Conlon's to show that for any odd number $k \geq 3$ and rational number $\frac{k-1}{k+1} \leq a<1$ there exists an unbalanced bipartite graph with $|A|=n,|B|=O_{k}\left(n^{a}\right)$ and $|E(G)|=\Omega_{k, a}\big(\left(n^{1+a}\right)^{\frac{k+1}{2 k}}\big)$ such that between any two vertices there exist at most $\ell$ paths of length $k$.

\section{Preliminaries}
To begin let $q$ be a prime and $\mathbb{F}_{q}$ a field of order $q$, we can take $\mathbb{F}_{q}=\mathbb{Z}_{q}$. We will talk about polynomials over $\mathbb{F}_{q}^{t}$ for a given natural number $t$, writing any such polynomial as $f(x)$ where $x=\left(x_{1}, \ldots, x_{t}\right) \in \mathbb{F}_{q}^{t}$. Let $d$ be a natural number, we define as $\mathbb{P}_{d}$ the set of polynomials in $X$ of degree at most $d$. That is, the set of linear combinations over $\mathbb{F}_{q}$ of monomials of the form $x_{1}^{a_{1}} x_{2}^{a_{2}} \ldots x_{t}^{a_{t}}$ with $\sum_{i=1}^{t} a_{i} \leq d$. By a random polynomial, we just mean a polynomial chosen uniformly from the set $\mathbb{P}_{d}$. One may produce such a random polynomial by choosing the coefficients of the monomials above to be random elements of $\mathbb{F}_{q}$.

Note that $\mathbb{P}_{d} \cong \mathbb{F}_{q}^{\frac{t^{d+1}-1}{t-1}}$. So, we have the probability space $(\Omega, \mathcal{F},\mathbb{P})$ where $\Omega=\mathbb{P}_{d}, \mathcal{F}=$ $\mathcal{P}\left(\mathbb{P}_{d}\right)$ and $\mathbb{P}$ the uniform probability distribution. It is obvious that every function which has $\mathbb{P}_{d}$ as domain is a random variable.

The following two lemmas are taken from \cite{conlon2019graphs}.

\begin{lemma}\label{lemma 2.1}
If $f$ is a randomly chosen polynomial from $\mathbb{P}_{d}$ then, for any fixed $x \in \mathbb{F}_{q}^{t}$ we have
$$
\mathbb{P}[f(x)=0]=\frac{1}{q}.
$$
\end{lemma}

\begin{lemma}\label{lemma 2.2}
Assume that $q>\left(\begin{array}{c}m \\ 2\end{array}\right)$ and $d \geq m-1$. Then, if $f$ is a randomly chosen polynomial from $\mathbb{P}_{d}$ and $x_{1}, \ldots, x_{m}$ are $m$ distinct points from $\mathbb{F}_{q}^{t}$, we have
$$
\mathbb{P}\left[f\left(x_{1}\right)=\ldots=f\left(x_{m}\right)=0\right]=\frac{1}{q^{m}}.
$$
\end{lemma}

To introduce the final tools that we are going to use we must first give the following definition.

\begin{definition}
Given a field $\mathbb{F}_q$ we denote its algebraic closure with $\overline{\mathbb{F}}_q$. A variety over $\overline{\mathbb{F}}_q$ is a set $W$ of the form
\begin{align*}
    W := \left\{x \in \overline{\mathbb{F}}^t_q : f_1(x) = ... = f_s(x) = 0\right\},
\end{align*}
where $f_1,...,f_s$ are polynomials with domain $\overline{\mathbb{F}}^t_q$ that take values in $\mathbb{F}^t_q$. If the coefficients of these polynomials are in $\mathbb{F}_q$ we say that $W$ is defined over $\mathbb{F}^t_q$ and write $W(\mathbb{F}_q) = W \cap \mathbb{F}^t_q$. Furthermore, we say that $W$ has complexity $M \in \mathbb{N}$ if $s,t$ and the degrees of $f_1,...,f_s$ are all bounded from $M$. Additionally, we say that a variety is absolutely irreducible if it is irreducible over $\overline{\mathbb{F}}_q$. Finally, the dimension $\operatorname{dim}W$ is the maximum integer $d$ such that there exists a chain of absolutely irreducible subvarieties of W of the form
\begin{align*}
    \emptyset \subset \{p\} \subset W_1 \subset W_2\hspace{0.1cm} ... \subset W_d \subset W,
\end{align*}
where $p$ is a point.
\end{definition}

The next is the known Lang--Weil bound, see \cite{lang1954number}.

\begin{lemma}\label{lemma 2.4}
Suppose that $W$ is a variety over $\overline{\mathbb{F}}_{q}$ of complexity at most $M$. Then
$$
\left|W\left(\mathbb{F}_{q}\right)\right|=O_{M}\left(q^{\operatorname{dim} W}\right) .
$$

Moreover, if $W$ is defined over $\mathbb{F}_{q}$ and absolutely irreducible, then
$$
\left|W\left(\mathbb{F}_{q}\right)\right|=q^{\operatorname{dim} W}\big(1+O_{M}\big(q^{-\frac{1}{2}}\big)\big).
$$
\end{lemma}

The result below is standard in algebraic geometry, one for example can see \citet{bump1998algebraic}.

\begin{lemma}\label{lemma 2.5}
Suppose that $W$ is an absolutely irreducible variety over $\overline{\mathbb{F}}_{q}$ of complexity at most $M$ and $\operatorname{dim} W \geq 1$. Then, for any polynomial $g: \overline{\mathbb{F}}_{q}^{t} \longmapsto \overline{\mathbb{F}}_{q}, W \subseteq\{x: g(x)=0\}$ or $W \cap\{x: g(x)=0\}$ is a variety of dimension less than $\operatorname{dim} W$.
\end{lemma}

The last preliminary result is taken again from \cite{conlon2019graphs}.

\begin{lemma}\label{lemma 2.6}
Suppose that $W$ is a variety over $\overline{\mathbb{F}}_{q}$ of complexity at most $M$ which is defined over $\mathbb{F}_{q}$. Then, there are $O_{M}(1)$ absolutely irreducible varieties $Y_{1}, \ldots, Y_{s}$, each of which is defined over $\mathbb{F}_{q}$ and has complexity $O_{M}(1)$, such that $\bigcup_{i=1}^{s} Y_{i}\left(\mathbb{F}_{q}\right)=W\left(\mathbb{F}_{q}\right)$.
\end{lemma}

For completeness we also remind Bertrand's Postulate.

\begin{gather*}
\text{For every natural number}\hspace{0.2cm} n>1\hspace{0.2cm} \text{there is at least one prime}\hspace{0.2cm} p \hspace{0.2cm}\text{such that}\\
n<p<2 n.
\end{gather*}







\section{Construction}
\begin{theorem}
Let $k \geq 3$ be an odd number and $a \in \mathbb{Q}$ where $\frac{k-1}{k+1} \leq a<1$. Then exists a natural number $c_{k}$ such that for every $n \in \mathbb{N}$ sufficiently large, exists a $\theta_{k, c_{k}}-$free bipartite graph $G=(A, B)$ with $|A|=n,|B|=O_{k}\left(n^{a}\right)$ and 
\begin{align}
|E(G)|=\Omega_{k, a}\big(\left(n^{1+a}\right)^{\frac{k+1}{2 k}}\big).
\end{align}
\end{theorem}

\begin{proof}
Let $q$ be a sufficient large prime, $k=2 x+1, a=\frac{\tau}{\lambda}, \gamma=x(\lambda+\tau), t=k(\lambda+\tau)$ and $d=k r$, with $\operatorname{gcd}(\tau, \lambda)=1$ and $r$ a natural number that will be determined later. We define the probability space $(\Omega, F, \mathbb{P})$ where $\Omega=\mathbb{P}_{d}^{\gamma}, F=\mathcal{P}\left(\mathbb{P}_{d}^{\gamma}\right)$ and $\mathbb{P}$ the uniform probability distribution. Let $f_{1}, \ldots, f_{\gamma}: \mathbb{F}_{q}{ }^{k \lambda} \times \mathbb{F}_{q}{ }^{k \tau} \longmapsto \mathbb{F}_{q}$ be independent random polynomials in $\mathbb{P}_{d}$. We construct the bipartite graph $G$ with $A=\mathbb{F}_{q}{ }^{k \lambda}$ and $B=\mathbb{F}_{q}{ }^{k \tau}$ where two vertices $u, v$ are connected if and only if $f_{i}(u, v)=0 \hspace{0.2cm} \forall i \in\{1,, \gamma\}$. Since $f_{1}, \ldots, f_{\gamma}$ were chosen independently, by \cref{lemma 2.2} the probability a given edge $(u, v)$ is in $G$ is $q^{-\gamma}$. Therefore $|E(G)|: \mathbb{P}_{d}^{\gamma} \longmapsto \mathbb{N}$ is a random variable and if we define $e_{(u, v)}: \mathbb{P}_{d}^{\gamma} \longmapsto\{0,1\}$, where $e_{(u, v)}\left(f_{1}, \ldots, f_{\gamma}\right) :=\left\{\begin{array}{ll}0 & (u, v) \notin E(G) \\ 1 & (u, v) \in E(G)\end{array}\right.$, the expected number of edges is

$$
\begin{gathered}
\mathbb{E}[|E(G)|]=\mathbb{E}\left[\sum_{(u, v) \in \mathbb{F}_{q}{ }^{k \lambda} \times \mathbb{F}_{q}{ }^{k \tau}} e_{(u, v)}\right]= \\
\sum_{(u, v) \in \mathbb{F}_{q}{ }^{k \lambda} \times \mathbb{F}_{q}{ }^{k \tau}} \mathbb{E}\left[e_{(u, v)}\right]=q^{k(\lambda+\tau)-\gamma} \geq q^{k(\lambda+\tau) \frac{k+1}{2 k}}.
\end{gathered}
$$

Suppose now that $w_{1}, w_{2}$ are two fixed vertices in $A, B$ respectively and let $S$ be the set of paths of length $k$ between them. We are going to estimate the $r$-th moment of $|S|$. At this point it is important to note that $|S|^{r}: \mathbb{P}_{d}^{\gamma} \longmapsto \mathbb{N}$, is a random variable that counts the number of ordered collections of $r$ paths of length $k$ in $G$ between $w_{1}, w_{2}$ without any restrictions, possibly overlapping or identical. Since the total number of edges $m$ in any such collection of $r$ paths is at most $d=k r$ and $q$ is sufficiently large prime \cref{lemma 2.2} tells us that the probability that particular collection of paths is in $G$ is $q^{-\gamma m}$, where we again used the fact that $f_{1}, \ldots, f_{\gamma}$ were chosen independently. Within the complete bipartite graph between $A$ and $B$, let $P_{r, m}$ be the number of ordered collections of $r$ paths, each of length $k$, from $w_{1}$ to $w_{2}$ whose union has $m$ edges. So define the random variables $|S|_{m}^{r}: \mathbb{P}_{d}^{\gamma} \longmapsto \mathbb{N}$ that counts exactly that and we have $$|S|^{r}=\sum_{m=1}^{k r}|S|_{m}^{r}.$$ We needed to distinguish the collection of paths by the number of the edges that they contain because the probability of their existence depends on $m$. Then as before, we get

$$
\mathbb{E}\left[|S|^{r}\right]=\sum_{m=1}^{k r} P_{r, m} q^{-\gamma m}.
$$

To finish the estimation we need to consider the size of $P_{r, m}$. Before that, fix some $m \in\{1, \ldots, k r\}$ and some $\left(n_{1}, n_{2}\right) \in\{1, \ldots, r x\} \times\{1, \ldots, r x\}$. Assume that there is a collection of $r$ paths of length $k$ from $w_{1}$ to $w_{2},\left(p_{1}, \ldots, p_{r}\right)$, that are defined from $n_{1}$ inner vertices of $A$ and $n_{2}$ inner vertices of $B$ and are constructed by $m$ edges. We want to show that $k \lambda n_{1}+k \tau n_{2} \leq \gamma m$ or equivalently

$$
(2 x+1) \lambda n_{1}+(2 x+1) \tau n_{2} \leq x(\lambda+\tau) m.
$$

For all $j \in\{1, \ldots, r\}$ let $m_{j}$ be the number of edges that belong to $p_{j} \backslash \bigcup_{i=1}^{j-1} p_{i}$ and similarly $n_{1, j}, n_{2, j}$ the number of inner vertices of A,B that belong to $p_{j} \backslash \bigcup_{i=1}^{j-1} p_{i}$. By definition we have $$\sum_{j=1}^{r} m_{j}=m, \sum_{j=1}^{r} n_{1, j}=n_{1}, \sum_{j=1}^{r} n_{2, j}=n_{2},$$ we will prove the desired inequality by proving that $\forall j \in\{1, \ldots, r\}$

$$
(2 x+1) \lambda n_{1, j}+(2 x+1) \tau n_{2, j} \leq x(\lambda+\tau) m_{j} .
$$

\textbf{Claim}\\
For a given $j \in\{1, \ldots r\}$ at least one of the following is true :
\begin{itemize}
    \item 
$\left\{m_{j} \geq 2 n_{1, j}+1\right.$ and $\left.m_{j} \geq 2 n_{2, j}+1\right\}$,
\item $\left\{m_{j} \geq 2 n_{1, j}+2\right.$ and $\left.m_{j} \geq 2 n_{2, j}\right\}$,
\item $\left\{m_{j} \geq 2 n_{1, j}\right.$ and $\left.m_{j} \geq 2 n_{2, j}+2\right\}$.
\end{itemize}

\begin{proof}[Proof of the claim]
To begin take the case where the new $m_{j}$ edges create a single path. There are 3 possibilities, in the first the path starts and ends in different parts of $G$ and we get the first set of inequalities. The second possibility is when the path starts at $A$ and finishes at $A$ and the
third possibility is when the path starts at $B$ and finishes at $B$, it is easy to see that in those we get the second and third set of inequalities respectively. Now consider the general case where the new $m_{j}$ edges create multiple disjoint paths. Again, it is fairly obvious that if we examine one of the paths and find that satisfies one of the sets of inequalities from the above analysis, these set remains true even after all paths are examined.

Note that $x \geq n_{1, j}, n_{2, j} \forall j \in\{1, \ldots, r\}$, hence in the first instance we have

$$
\begin{aligned}
x(\lambda+\tau) m_{j} = & x \lambda m_{j}+x \tau m_{j} \geq x \lambda\left(2 n_{1, j}+1\right)+x \tau\left(2 n_{2, j}+1\right) \\
& \geq(2 x+1) \lambda n_{1, j}+(2 x+1) \tau n_{2, j} .
\end{aligned}
$$

In the second instance, because $\lambda>\tau$ we have

$$
\begin{aligned}
x(\lambda+\tau) m_{j} & =x \lambda m_{j}+x \tau m_{j} \geq x \lambda\left(2 n_{1, j}+2\right)+x \tau 2 n_{2, j} \\
& \geq(2 x+1) \lambda n_{1, j}+(2 x+1) \tau n_{2, j} .
\end{aligned}
$$

In the final instance we have

$$
\begin{gathered}
x(\lambda+\tau) m_{j}=x \lambda m_{j}+x \tau m_{j} \geq x \lambda 2 n_{1, j}+x \tau\left(2 n_{2, j}+2\right) \\
=2 x \lambda n_{1, j}+(2 x+1) \tau n_{2, j}+\left(2 x-n_{2, j}\right) \tau .
\end{gathered}
$$

Before we continue we need to consider the two possible subcases. The first is that the new $m_{j}$ edges create a single path that starts at $B$ and finishes at $B$ and the second is when the new $m_{j}$ edges create multiple disjoint paths. In the second subcase if one of the paths starts and ends in different parts of $G$ or starts and ends at $A$ we can go to one of the first two instances and we are done. If they create multiple paths that start and end in $B$ then in the above string of inequalities we can substitute 2 by 4 and get $$x(\lambda+\tau) m_{j} \geq 2 x \lambda n_{1, j}+(2 x+1) \tau n_{2, j}+3 x \tau,$$ but $3 \tau \geq \lambda$, therefore again we get what we want. Finally we need only to consider the first subcase. So now we have $n_{1, j}=\frac{m_{j}}{2}, n_{2, j}+1=\frac{m_{j}}{2}$ and want to prove that $\left(2 x-n_{2, j}\right) \tau \geq \lambda n_{1, j}$, equivalently

$$
a \geq \frac{n_{1, j}}{2 x-n_{2, j}}=\frac{\frac{m_{j}}{2}}{2 x+1-\frac{m_{j}}{2}}=\frac{1}{\frac{2 x+1}{\frac{m_{j}}{2}}-1}.
$$

But for all $j$ we have $m_{j} \leq 2 x$, so

$$
\frac{1}{\frac{2 x+1}{\frac{m_{j}}{2}}-1} \leq \frac{1}{\frac{2 x+1}{x}-1}=\frac{x}{x+1}=\frac{k-1}{k+1}.
$$
\end{proof}
Denote with $\Gamma_{m}$ the pairs $\left(n_{1}, n_{2}\right) \in\{1, \ldots, r x\} \times\{1, \ldots, r x\}$ that there exists a collections
of $r$ paths of length $k$ from $w_{1}$ to $w_{2},\left(p_{1}, \ldots, p_{r}\right)$, that are defined from $n_{1}$ inner vertices of $A$ and $n_{2}$ inner vertices of $B$ and are constructed by $m$ edges. For every such pair $\left(n_{1}, n_{2}\right)$ there are at most $q^{k \lambda n_{1}+k \tau n_{2}}=q^{(2 x+1) \lambda n_{1}+(2 x+1) \tau n_{2}}$ such collections, so

$$
P_{r, m} \leq \sum_{\left(n_{1}, n_{2}\right) \in \Gamma_{m}} q^{(2 x+1) \lambda n_{1}+(2 x+1) \tau n_{2}} .
$$

Note that $\left|\Gamma_{m}\right| \leq x^{2} r^{2}$ for every $m \in\{1, \ldots, r k\}$. Now we return to the expectation of $|S|^{r}$ and get

$$
\begin{aligned}
\mathbb{E}\left[|S|^{r}\right]= & \sum_{m=1}^{k r} P_{r, m} q^{-\gamma m} \leq \sum_{m=1}^{k r} \sum_{\left(n_{1}, n_{2}\right) \in \Gamma_{m}} q^{(2 x+1) \lambda n_{1}+(2 x+1) \tau n_{2}-x(\lambda+\tau) m} \\
& \leq \sum_{m=1}^{k r}\left|\Gamma_{m}\right| \leq \sum_{m=1}^{k r} x^{2} r^{2}=k x^{2} r^{3}=k\left(\frac{k-1}{2}\right)^{2} r^{3} .
\end{aligned}
$$

By Markov's inequality we conclude that $\mathbb{P}[|S| \geq s]=\mathbb{P}\left[|S|^{r} \geq s^{r}\right] \leq k\left(\frac{k-1}{2}\right)^{2} \frac{r^{3}}{s^{r}}$. We need to note that the set of paths $S$ is a subset of
\begin{gather*}
T := \left\{\left(x_{1}, \ldots, x_{k-1}\right): x_{2 i-1} \in \mathbb{F}_{q}{ }^{k \tau}, x_{2 i} \in\right.\mathbb{F}_{q}{ }^{k \lambda} \hspace{0.2cm} \forall i \in\{1, \ldots, x\}\hspace{0.2cm} \text{and} \\ \left.f_{j}\left(w_{1}, x_{1}\right)=f_{j}\left(x_{1}, x_{2}\right)=\ldots=f_{j}\left(x_{k-1}, w_{2}\right)=0, \hspace{0.2cm} \forall j \in\{1, \ldots, \gamma\}\right\}.
\end{gather*}
Unfortunately $T$ may contain degenerate walks as well as paths we are interested in, so we proceed to the following analysis. If $T$ contains a degenerate walk $w_{1}, x_{1}, \ldots, x_{k-1}, w_{2}$ it must be one of the following cases, either $w_{1}=x_{b}$ for some $1 \leq b \leq k-1$ or $x_{a}=x_{b}$ for some $1 \leq a<b \leq k-1$ or $x_{a}=w_{2}$ for some $1 \leq a \leq k-1$. Let us then define the sets :

\begin{itemize}
  \item $T_{0 b}=T \cap\left\{\left(x_{1}, \ldots, x_{k-1}\right): w_{1}=x_{b}\right\}$ for some $1 \leq b \leq k-1$,

  \item $T_{a b}=T \cap\left\{\left(x_{1}, \ldots, x_{k-1}\right): x_{a}=x_{b}\right\}$ for some $1 \leq a<b \leq k-1$,

  \item $T_{a k}=T \cap\left\{\left(x_{1}, \ldots, x_{k-1}\right): x_{a}=w_{2}\right\}$ for some $1 \leq a \leq k-1$.

\end{itemize}

Since $T$ is defined over $\mathbb{F}_{q}$ and has complexity bounded in terms of $k$, \cref{lemma 2.6} tells us that there are $O_{k}(1)$ absolutely irreducible varieties $Y_{1}, \ldots, Y_{s}$, each of which is defined over $\mathbb{F}_{q}$ and has complexity $O_{k}(1)$, such that $\bigcup_{i=1}^{s} Y_{i}\left(\mathbb{F}_{q}\right)=T\left(\mathbb{F}_{q}\right)$. If $\operatorname{dim} Y_{i} \geq 1$, \cref{lemma 2.5} tells us that either there exist $a$ and $b$ such that $Y_{i} \subseteq T_{a b}$ or the dimension of $Y_{i} \cap T_{a b}$ is smaller than the dimension of $Y_{i}$ for all $a$ and $b$. If $Y_{i} \subseteq T_{a b}$ for some $a$ and $b$, the component does not contain any non-degenerate paths and may be removed from consideration. If instead the dimension of $Y_{i} \cap T_{a b}$ is smaller than the dimension of $Y_{i}$ for all $a$ and $b$, the Lang-Weil bound, \cref{lemma 2.4}, tells us that for $q$ sufficiently large

$$
|S| \geq\left|Y_{i}\left(\mathbb{F}_{q}\right)\right|-\sum_{a, b}\left|Y_{i} \cap T_{a b}\right| \geq q^{\operatorname{dim} Y_{i}}-O_{k}\left(q^{\operatorname{dim} Y_{i}-\frac{1}{2}}\right)-O_{k}\left(q^{\operatorname{dim} Y_{i}-1}\right) \geq \frac{q}{2}.
$$

On the other hand, if $\operatorname{dim} Y_{i}=0$ for every $Y_{i}$ which is not contained in some $T_{a b}$, \cref{lemma 2.4} tells us that $|S| \leq \sum\left|Y_{i}\right|=O_{k}(1)$, where the sum is taken over all $i$ for which $\operatorname{dim} Y_{i}=0$. Putting everything together, we see that that there exists a constant $c_{k}-1$, depending only on $k$, such that either $|S| \leq c_{k}-1$ or $|S| \geq q / 2$. Therefore, by the consequence of Markov's inequality noted earlier we get

$$
\mathbb{P}\left[|S|>c_{k}-1\right]=\mathbb{P}\left[|S| \geq \frac{q}{2}\right] \leq k\left(\frac{k-1}{2}\right)^{2} \frac{r^{3}}{\left(\frac{q}{2}\right)^{r}}.
$$

Define a pair of vertices $\left(w_{1}, w_{2}\right)$ bad if $w_{1} \in A, w_{2} \in B$ and there are more than $c_{k}-1$ paths of length $k$ between them. Let $\Lambda: \mathbb{P}_{d}^{\gamma} \longmapsto \mathbb{N}$ be the random variable counting the number of bad pairs. So

$$
\mathbb{E}[\Lambda] \leq \sum_{\left(w_{1}, w_{2}\right) \in A \times B} \mathbb{P}\left[|S|>c_{k}-1\right]=q^{k(\lambda+\tau)} k\left(\frac{k-1}{2}\right)^{2} \frac{r^{3}}{\left(\frac{q}{2}\right)^{r}}.
$$

Now remove a vertex from $B$ for every bad pair. Since each vertex has degree at most $n$ the total number of edges removed is at most $\Lambda n$. From all the above the expected number of edges for the final graph $G^{\prime}$ would be

$$
\mathbb{E}\left[\left|E\left(G^{\prime}\right)\right|\right] \geq q^{k(\lambda+\tau) \frac{k+1}{2 k}}-q^{k(\lambda+\tau)+k \lambda-r} k\left(\frac{k-1}{2}\right)^{2} \frac{r^{3}}{\left(\frac{1}{2}\right)^{r}}.
$$

Hence, if we define $$r :=\frac{3 k-1}{2} \lambda+\frac{k-1}{2} \tau+1,$$ there are choices of $f_{1}, \ldots, f_{\gamma}$ with the desired properties. As stated, this result only holds when $q$ is a prime and $n=q^{k l}$. To generalize this let $n$ be a sufficiently large natural number, then by Bertrand's postulate we have that exists a prime $p_{n}$ such that

$$
\begin{gathered}
\left\lfloor\frac{n^{\frac{1}{k l}}}{2}\right\rfloor<p_{n}<2\left\lfloor\frac{n^{\frac{1}{k l}}}{2}\right\rfloor \Rightarrow \frac{n^{\frac{1}{k l}}}{4}<p_{n}<n^{\frac{1}{k l}} \Rightarrow \\
\frac{1}{4^{k l}} n<p_{n}^{k l}<n .
\end{gathered}
$$

Hence, by applying the previous result to $p_{n}$ and adding to part $A$ the remaining vertices until we reach $|A|=n$ we complete the proof. 
\end{proof}

From the techniques of the proof it is obvious that a somewhat stronger result was proved because the $c_{k}$ paths need not to be internally disjoint. But that is counterbalanced by the fact that the value of $c_{k}$ becomes undesirable big in correlation with $k$.

\section{Conclusion}
A direct consequence from our construction is that $\operatorname{ex}\left(n^{a}, n, \theta_{k, c_{k}}\right)=\Omega_{k, a}\big(\left(n^{1+a}\right)^{\frac{k+1}{2 k}}\big)$. Combining this with \citet{jiang_ma_yepremyan_2022}[Theorem 1.11] which states that $\operatorname{ex}\left(n^{a}, n, \theta_{k, c_{k}}\right) = O_{k, a}\big(\left(n^{1+a}\right)^{\frac{k+1}{2 k}}+$ $n+n^{a}\big)$  we conclude to our final result.

\begin{theorem}
Let $k \geq 3$ be an odd number and $a \in \mathbb{Q}$, where $\frac{k-1}{k+1} \leq a<1$. Then, there exists a natural number $c_{k}$ such that for every $n \in \mathbb{N}$ sufficiently large we have
\begin{align}
\operatorname{ex}\left(n^{a}, n, \theta_{k, c_{k}}\right)=\Theta_{k, a}\big(\left(n^{1+a}\right)^{\frac{k+1}{2 k}}\big).
\end{align}
\end{theorem}

\section*{Acknowledgements}
I would like to sincerely thank Dr. Hong Liu for his valuable comments in the process of writing this paper.


\begin{thebibliography}{6}
\providecommand{\natexlab}[1]{#1}
\providecommand{\url}[1]{\texttt{#1}}
\expandafter\ifx\csname urlstyle\endcsname\relax
  \providecommand{\doi}[1]{doi: #1}\else
  \providecommand{\doi}{doi: \begingroup \urlstyle{rm}\Url}\fi

\bibitem[Bukh(2014)]{article}
Boris Bukh.
\newblock Random algebraic construction of extremal graphs.
\newblock \emph{Bulletin of the London Mathematical Society}, 47, 09 2014.
\newblock \doi{10.1112/blms/bdv062}.

\bibitem[Bump(1998)]{bump1998algebraic}
Daniel Bump.
\newblock \emph{Algebraic geometry}.
\newblock World Scientific Publishing Company, 1998.

\bibitem[Conlon(2019)]{conlon2019graphs}
David Conlon.
\newblock Graphs with few paths of prescribed length between any two vertices.
\newblock \emph{Bulletin of the London Mathematical Society}, 51\penalty0 (6):\penalty0 1015--1021, 2019.

\bibitem[Faudree and Simonovits(1983)]{faudree1983class}
Ralph~J Faudree and Mikl{\'o}s Simonovits.
\newblock On a class of degenerate extremal graph problems.
\newblock \emph{Combinatorica}, 3:\penalty0 83--93, 1983.

\bibitem[Jiang et~al.(2022)Jiang, Ma, and Yepremyan]{jiang_ma_yepremyan_2022}
Tao Jiang, Jie Ma, and Liana Yepremyan.
\newblock On turán exponents of bipartite graphs.
\newblock \emph{Combinatorics, Probability and Computing}, 31\penalty0 (2):\penalty0 333–344, 2022.
\newblock \doi{10.1017/S0963548321000341}.

\bibitem[Lang and Weil(1954)]{lang1954number}
Serge Lang and Andr{\'e} Weil.
\newblock Number of points of varieties in finite fields.
\newblock \emph{American Journal of Mathematics}, 76\penalty0 (4):\penalty0 819--827, 1954.

\end{thebibliography}

\end{document}